\documentclass[12pt]{article}

\usepackage[margin=1in]{geometry}
\usepackage[backend=biber]{biblatex}
\nocite{*}

\usepackage{color}
\usepackage{amsfonts, amsmath, amsthm, amssymb}
\usepackage{latexsym, amsxtra, amscd, ifthen}
\usepackage{dsfont}
\usepackage{verbatim}
\usepackage{url}
\usepackage{titlesec}
\usepackage{graphicx}
\usepackage{tikz, pgfplots}
\usepackage{tikz-cd, pst-node}
\pgfplotsset{compat=1.18}
\usetikzlibrary{positioning, arrows.meta}

\tikzset{
    STEPBOX/.style={
    rectangle, 
    draw=black, 
    fill=white, 
    very thick, 
    minimum height=7mm,
    minimum width=40mm,
    }
}

\graphicspath{ {./images/} }

\titleformat{\section}[block] 
  {\normalfont\Large\filcenter} 
  {\thesection.}{0.5em}{}          

\titleformat{\subsection}[runin]
    {\normalfont\bfseries}
    {\thesubsection.}{0.5em}{}
    [\newline]

\newtheorem{theorem}{Theorem}
\newtheorem{lemma}[theorem]{Lemma}
\newtheorem{proposition}[theorem]{Proposition}

\newtheorem{corollary}[theorem]{Corollary}

\newtheorem{definition}[theorem]{Definition}
\newtheorem{notation}[theorem]{Notation}
\newtheorem{example}[theorem]{Example}
\newtheorem{remark}[theorem]{Remark}

\numberwithin{theorem}{section}
\numberwithin{equation}{theorem}

\newcommand{\kk}{\mathds{k}} 

\title{Foundations and Classification of Invariant Subalgebras of Grassmann Algebra}
\author{Mithat Konuralp Demir\\
\small Mentor: Zahra Nazemian}
\date{August 2025}

\begin{document}

\maketitle

\begin{abstract}
This paper is a documentation of author's reseach, focusing on the topic Grassmann Algebra spanning over July, August 2025 under mentorship provided by DRP Türkiye 2025.

Grassmann algebra is a fundamental structure in mathematics with wide-ranging applications across multiple areas of mathematics and physics. Most notably, it serves as the foundation for differential geometry, by constituting the natural setting which differential forms reside. This paper begins with presenting the defining properties of Grassmann Algebra, outlining the working principles of the key mechanism of the algebra, wedge product. Following that, we give an exposition of formal construction of Grassmann algebra from free associative algebra with the goal of emphasizing how these properties are imposed in the structure of the algebra. The intrinsic relationship between the exterior product and the determinant is explored in Section 4. Finally, we investigate invariant subalgebras, one of the primary focuses of this paper. Here, we present a novel classification of invariant subalgebras. 
\end{abstract}

\tableofcontents


\section{Introduction}

\hspace*{\parindent} \emph{Grassmann algebra} is a major discovery in mathematics of the 19th century, conceived by German mathematician and gymnasium teacher Hermann Grassmann. It is referred to, synonymously, as \emph{Exterior algebra}. Grassmann presents this algebra for the very first time in his prominent work \emph{Ausdehnungslehre} dating back to 1844, which is \emph{Expansion Theory} in English. The name Exterior algebra originates from the geometric interpretation of the algebra: an element of the algebra may be \emph{“extended”} to form a higher order element by its (exterior) product with another element, which will become clearer in Section 2. We shall refer to the algebra as Grassmann algebra in this introduction, which is a brief discussion on the roots and history of the algebra. In contrast, in the following parts of the text, we adopt the name \emph{Exterior algebra}, which appears to be more suitable, provided we focus on the internal algebraic structure of this algebra and its intrinsic connection with geometry.

This paper reports the author’s introductory research on Grassmann algebra. In Section 2, we introduce the central machinery of the algebra, the \emph{wedge product}, and exemplify its applications in Euclidean geometry. In Section 3, we construct the algebra from free associative (non-commutative) algebra step by step and compare it to the similar construction of polynomial rings, with the goal of disclosing how the most important inherent property of the algebra, anti-commutativity (or anti-symmetry) is implemented into it. Note that these rings are algebras over free associative algebra on $n$-variables. In Section 4, we discuss the fundamental connection between Exterior algebra and the determinant. Final section explores with much more rigor some algebraic properties of the algebra, mainly the automorphisms of the algebra and subalgebras invariant under them.

Grassmann algebra is one of those curious cases in mathematics that received the recognition and appreciation it is worthy of, long after its conception. Hermann Grassmann spent a considerable amount of effort to promote his ideas, and although application of his algebra to mechanics was known by the 19th century, it did not spark much interest in the mathematical community. John Browne, in his comprehensive work \cite{browne}, comments on this strange phenomenon by arguing that “the mathematics at that time was not ready for the new ideas Grassmann proposed, and now in the twenty-first century, seems only just becoming aware of their potential”. 

In the late 19th century, there were major advancements regarding Grassmann algebra. One of them is Peano’s reformulation of Grassmann’s ideas in a more modern and axiomatic way which he was famous for. He is one of the few who recognized the depth of Grassmann’s work and his reformulation made Grassmann’s theory more mainstream and accessible to other mathematicians. If it were not for Peano, the theory might have remained longer in obscurity. Another great progress is Cartan’s application of exterior algebra to differential geometry. By defining the exterior derivative $d$, thereby equipping the algebra of differential forms with a calculus, he uncovered much of the potential of Grassmann algebra. This revolutionary step rendered the algebra indispensable for geometry. 

Today, the prominence of this algebra visibly spans a wide range of scientific fields concerned with geometry, most notably mathematics and physics. Therefore, it appears that more mathematicians will undertake the effort to study and further develop Grassmann algebra, and it will receive greater celebration in the mathematical community.

\subsection{Preliminaries}
We start by outlining some preliminary definitons.
For our purposes, we consider vector spaces and algebras defined over only algebraically closed field of characteristic zero. In this case, let $\kk$ be such a field.
\begin{definition}
    \textbf{Binary operation} $*$ on a non-empty set S is a map defined as $ * : S \times S \to S$ 
\end{definition}
\begin{definition}
    Let $S$ be a set and $* : S \times S \to S$ be a binary operation. We say this operation is \textbf{associative} if $a*(b*c) = (a*b)*c$, and \textbf{commutative} if $a*b = b*a$, for all $a, b, c \in S$.
\end{definition}
\begin{definition}
    A \textbf{group} is a tuple $(S, *)$, where $S$ is a set and $*$ is a binary product on $S$ which satisfies the following 4 axioms:
    \begin{itemize}
        \item Closure: $*$ is closed on $S$
        \item Associativity: $*$ is associative
        \item Identity: an identity element $id$ exists
        \item Inverses: for every element $a \in S$, there exists an inverse $b \in S$ such that $ab = ba = id$
    \end{itemize}
\end{definition}
\begin{definition}
    \textbf{Ring} is a triple $(R, +, \cdot)$ such that $R$ is a set equipped with two operations, namely addition $+$ and multiplication $\cdot$, that is an abelian(commutative) group under addition $+$, multiplication $\cdot$ is associative and distributes over addition.
\end{definition}
\begin{definition}
    Let R be a ring with multiplication $\cdot$. A \textbf{two-sided} ideal $I \subseteq R$ is a subset that is an additive subgroup of the group $(R, +)$ that satisfies for any $r \in R$ and $x \in I$, $r \cdot x \in I$ and $x \cdot r \in I$
\end{definition}
\begin{definition}
    $\kk$\textbf{-vector space} is the triple $(V, \kk, +, \cdot)$ where $V$ is a set equipped with compatible binary operations $+$, $\cdot$ that are called \textbf{vector addition}, \textbf{scalar multiplication}. It is an abelian group under addition $+$, and multiplication $\cdot$ satisfies, for all $u, v \in V$ and $a, b \in \kk$, $a(u+v) = au + av,\quad (a+b)u = au + bu,\quad (ab)u = a(bu),\quad 1u = u$. 
\end{definition}
\begin{definition}
    Let $V, W, U$ be vector spaces over the same field $\kk$. A map $\varphi : V \times W \to U$ is \textbf{bilinear} if for all $v, v' \in V$ and $w, w' \in W$ and $a \in \kk$ we have;
    \begin{gather*}
        (av) \cdot w = a(v \cdot w)\\
        v \cdot (aw) = a(v \cdot w)\\
        v \cdot (w + w') = v \cdot w + v \cdot w'\\
        (v + v') \cdot w = v \cdot w + v' \cdot w
    \end{gather*}
\end{definition}
\begin{definition}
    $\kk$\textbf{-algebra} $A$ is a vector space over $\kk$ equipped with a bilinear multiplication operation $\cdot : A \times A \to A$ that is closed under multiplication.
\end{definition}
\begin{definition}
    Let $\cdot : V \times V \to W$ be bilinear map. We say this operation is:
    \begin{itemize}
        \item \textbf{Alternating} (nilpotent), if $v \cdot v = 0$
        \item \textbf{Anti-commutative} (anti-symmetric), if $u \cdot v = -v \cdot u$
    \end{itemize}
    holds for all $u, v \in V$
\end{definition}
\begin{remark}\label{rmk:altanti}
    For $char(\kk) \neq 2$, we have equivalence of alternating and anti-commutativity properties for a bilinear map $\cdot$. Specifically, $\cdot$ is alternating, then $0 = (v + w)(v + w) = v^2 + vw + wv + w^2 = vw + wv$, thus anti-commutativity follows. It is easier to see how alternating property follows from anti-commutativity. 

    On the other hand, in $char(\kk) = 2$, alternating still implies anti-commutativity. However, anti-commutativity is equivalent to commutativity as $vw = -wv = wv$ ($-1 = 1$ in $char = 2$) and we can have $v^2 \neq 0$ but still $vw = -wv$.

    Thus, even though adjectives alternating and anti-commutativity are used interchangeably, they are identical only conditionally, and must be used with care.
\end{remark}

\begin{definition}\label{grading}
    When $G$ is a group,  a $\kk$-algebra $A$ is called \textbf{ $G$-graded} if it can be decomposed into a direct sum of $\kk$-vector subspaces, such as:
    $$A = \bigoplus_{i \in G} A_i$$
    where $i$ runs over $G$ and  
    $$A_g \cdot A_h \subseteq A_{gh} \quad \forall g, h \in G $$
    Whenever $A$ is $ \Bbb{Z}$-graded and $ A_i = 0$, for every $ i < 0$, we simply call $A$  a  \textbf{graded algebra}.  
\end{definition}
\begin{definition}
    A \textbf{connected graded algebra} is a graded algebra defined as in \ref{grading} where $A_0 = \kk$
\end{definition}
\begin{definition}
    \textbf{General Linear group} over a vector space $V$, or \textbf{$GL(V)$} is defined to be the group of all non-singular matrices over V under matrix multiplication. In other words, it is the group of all invertible linear transformations from $V$ to itself, under composition of maps.
\end{definition}
\begin{definition}
    Let $V$ be a vector space over field $\kk$. A map $\varphi : V^n \to \kk$ is said to be \textbf{multilinear} if for the $i^{th}$ argument (where $1 \le i \le n$), following holds: 
    $$
        \varphi(\cdots, av + bw, \cdots) = a\;\varphi(\cdots, v, \cdots) + b\;\varphi(\cdots, w, \cdots)
    $$
    for arbitrary $v, w \in V$ and $a, b \in \kk$
\end{definition}

\pagebreak
\begin{definition}
    The \textbf{free associative algebra} $R\langle X \rangle$ on a set $X$ over a commutative ring $R$ is the non-commutative, associative, unital $R$-algebra consisting of all finite $R$-linear combinations of formal words in $X$. Multiplication is defined by concatenation (writing symbols adjacently), and extended linearity: it distributes over addition and scalar multiplication.
\end{definition}
\begin{remark}\label{free}
    An algebra is $\textbf{free}$ if it has no relations other than the defining ones of its type (like associativity, unit, commutativity, e.t.c.) 
    The adjective "free" in this context means that the algebra has no relations other than the defining ones of its type (like associativity, unit, commutativity, e.t.c.). For example a \emph{free associative algebra} has no imposed relations expect associativity and unit, while a \emph{free commutative algebra} has no relations except associativity, commutativity and unit. In fact, the notion of \emph{free} depends on the category(type of algebra) we are considering.
    (That is why we have \emph{freeness} for an associative object, also a different \emph{freeness} for a commutative object). The origins of this notion lie in category theory, where a free object is characterized by a universal property. Treatment of category theory falls outside the scope of this paper; therefore we will not go deep into it, rather be content with an algebraic version of it. An interested reader may wish to investigate this subject in the context of category theory.
\end{remark}
\begin{remark}
    The multiplication is $R$-bilinear:
    for all $a, b \in R, \; f, g, h \in R\langle X \rangle$, we have
    \begin{gather*}
        (af + bg)h = a(fh) + b(gh)\\
        h(af + bg) = a(hf) + b(hg)
    \end{gather*}
\end{remark}

\pagebreak

\section{Exterior Algebra}

Exterior algebra is a powerful mathematical system for describing the physical world by providing an algebraic framework for vectors and tensors. One of its greatest strengths is the ability to algebraically encode linear dependence of entities. The source of this strength is the built-in product operation of the algebra, namely the \emph{wedge product}. 

\begin{notation}
    Wedge product is synonymously called \emph{exterior product}. In both cases, it is denoted with the wedge symbol $\wedge$
\end{notation}

\subsection{Definition of Exterior Algebra}
Before giving the definition and properties of the wedge product, we introduce the concept of \emph{multivector}. The exterior algebra, $\bigwedge^*(V)$ is an algebra over a vector space $V$ with the wedge product assigned as the multiplication operation.

For the following discussion, we suppose $V$ is a $n$-dimensional vector space with a fixed ordered basis $B = \{ e_1, \cdots, e_n \}$.

\begin{notation}
    Throughout the text we denote the exterior algebra over an arbitrary vector space $\bigwedge^*(V)$ with letter $A$.
    In this convention, $A_k$ denotes the subspace:
    $$A_k := \bigwedge^kV = \underbrace{V \wedge V \wedge \cdots \wedge V}_{k \text{ many}}$$
\end{notation}
We define the set of direct sum of subspaces of $A$ of even grade elements, and of odd grade elements as follows: 
\begin{definition}
    \begin{align*}
        A_{even} &= \bigoplus_{i = 0}^{\left\lfloor \frac{n}{2} \right\rfloor} A_{2i}\\
        A_{odd} &= \bigoplus_{i = 1}^{\left\lceil \frac{n}{2} \right\rceil} A_{2i+1}
    \end{align*}
\end{definition}
\noindent
We know a direct sum of subspaces is a subspace. Consequently, $A_{even}$ and $A_{odd}$ are subspaces of A. 

There are two natural gradings of exterior algebra. One of them is the $\mathbb{Z}_2$ grading:
$$A = A_{even} \oplus A_{odd}$$

The other grading is the $\mathbb{N}$ grading:
$$
A = \bigoplus_{i = 1}^n \bigwedge^i A_i
$$
where $A_0 = \kk, A_1 = V$. In perspective of $\mathbb{N}$ grading, exterior algebra is a connected graded algebra. 

\subsection{Wedge Product}
Of course, the wedge product is the main machinery that gives the algebra its function. We list the definition and main properties of it below:

\begin{definition}
    Let $V$ be a vector space over a field $\kk$ The \textbf{wedge product} is an associative, bilinear, alternating operation 
    $$\wedge : A_k \times A_l \to A_{k+l}$$
\end{definition}

\begin{remark}
    The \emph{alternating} property translates to \textbf{linear dependence}, which will be more apparent in an example with an element of higher grade.
\end{remark}
\begin{remark}
    Recall that, by \ref{rmk:altanti} the adjective \emph{alternating} in the definition implies \emph{anti-commutativity} for $char(\kk) \neq 2$, such that 
    $$x \wedge y = -y \wedge x, \quad \forall x, y \in V$$
    We will see that anti-commutativity is vital in encoding orientation of objects in geometry.
\end{remark}

Before proceeding further, we need a complete description of subspace $A_k$. 
\begin{definition}
    Any element of $A_k$ is called an element or \textbf{multivector of grade $k$},\\ \textbf{$k$-vector} or \textbf{$k$-element.} Specially, the term \emph{vector} is reserved for elements of $A_1$ which are 1-elements.
\end{definition}
Let $x, y$ be elements of $A_1 = V$. The $x \wedge y$ does not belong to $V$. Put differently, it is not a vector, but instead an instance of a multivector. In particular, a \emph{bivector} for it has two factors with respect to the wedge product. Simultaneously, according to the naming convention introduced above, it is an element of grade $2$. It belongs to the subspace $A_2$.

To help us in transcription, we introduce multi-index notation.

\begin{notation}
    An ordered \textbf{multi-index} (or simply multi-index) $I$ of grade k, or $k$-multi-index in $\mathbb{N}^n$ is an $n$-tuple of non-negative integers such that
    $$
    \begin{gathered}
    I = (i_1, i_2, \cdots , i_n) \\    
    \text{where } i_j \in \mathbb{N} \text{ with ordering } 1 \le i_1 < i_2 < \cdots <i_k \le n.
    \end{gathered}
    $$
    We explicitly state \emph{unordered}, when entries $i_j$ of multi-index $I$ are unordered.
\end{notation}
\begin{example}
    Say we have a set of vectors $S = \{v_1, v_2, \cdots, v_n\}$ and a $k$-multi-index $I = (i_1, i_2, \cdots, i_k)$ 
    In this context, wedge product of vectors in S, indexed by $I$ is 
    $$v_I = v_{i_1} \wedge v_{i_2} \wedge \cdots \wedge v_{i_k}$$
\end{example}

\begin{definition} 
Let $c \in \kk$, and let $x = e_1 \wedge \cdots \wedge e_k$ where $e_i \in B$. Then $c\,x$ is called an \textbf{irreducible element} of grade $k$.
\end{definition}

\begin{proposition}\label{prop:nil}
    Let $x = v_1 \wedge \cdots \wedge v_k$, where $v_i \in V$. If any vector appears more than once (up to scalar multiple), then $x = 0$
\end{proposition}
\begin{proof}
    Suppose, in the product, $\frac{1}{a}\;v_i = \frac{1}{b}\;v_j$ for some $a, b \in \kk$ and $i \neq j$. Reorder using anti-commutativity to bring $v_i$ next to $v_j$. Then, isolate the two vectors with associativity. From alternating property, $v_i \wedge v_j = 0$. Hence, $x = 0$.
\end{proof}

\begin{proposition}\label{prop:basis}
    Let $B_k = \{e_I = e_{i_1} \wedge \cdots \wedge e_{i_k} \mid \forall \;k\text{-multi-indices } I\}$. Here, the set of possible $k$-multi-indices $I$ is equivalent to all irreducible elements of grade $k$, and $I$ runs over all strictly increasing $k$-tuples of indices from $\{1, \cdots, n\}$.
    $B_k$ is a basis for the exterior $k^{th}$ exterior power, the subspace $A_k$. Therefore, $$dim(A_k) = \binom{n}{k}$$ 
\end{proposition}
\begin{proof}
    Proof follows easily from the definition $A_k := \bigwedge^kV$. For any element $v \in A_k$, $v$ is a linear combination irreducible elements of grade $k$. We require distinct vectors due to \ref{prop:nil}. Distributing wedge product over linear combinations, we obtain a linear combination of elements in $B_k$
\end{proof}
This proposition also establishes the dimension of the exterior algebra $A$:
\begin{proposition}
    $dim(A) = 2^n$, where $dim(V) = n$ 
\end{proposition}
\begin{proof}
    $$dim(A) = dim\left(\sum_{i=1}^nA_i\right) = \sum_{i=1}^n dim(A_i) = \sum_{i=1}^n \binom{n}{i} = 2^n$$
\end{proof}

\begin{example}
    Let $dim(V) = 4$ and $B = \{e_1, e_2, e_3, e_4\}$ be a basis for $V$. Then, the followings constitute a basis for the corresponding subspaces:
    \begin{itemize}
        \item $B_1 = B$, for $A_1$ ($dim(A_1) = 4$)
        \item $B_2 = \{e_1 \wedge e_2, e_1 \wedge e_3, e_1 \wedge e_4, e_2 \wedge e_3, e_2 \wedge e_4, e_3 \wedge e_4\}$, for $A_2$ ($dim(A_2) = 6$)
        \item $B_3 = \{e_1 \wedge e_2 \wedge e_3, e_1 \wedge e_2 \wedge e_4, e_1 \wedge e_3 \wedge e_4, e_2 \wedge e_3 \wedge e_4\}$, for $A_3$ ($dim(A_3) = 4$)
        \item $B_4 = \{e_1 \wedge e_2 \wedge e_3 \wedge e_4\}$, for $A_4$ ($dim(A_4) = 1$)

    \end{itemize}
\end{example}

With multi-index notation at our disposal to describe elements of higher grades, we wish to demonstrate how the properties of the wedge product extend multivectors of arbitrary grades.

\begin{proposition}
    Let $v$ be a wedge product of vectors in $V$, such that $$v = \cdots \wedge x \wedge \cdots \wedge y \wedge \cdots \quad \text{ where } x, y \in V$$ By interchanging the order of vector factors $x$ and $y$ we change the sign of the product $v$, which is $$v' = -v = \cdots \wedge y \wedge \cdots \wedge x \wedge \cdots$$
\end{proposition}
\begin{proof}
    By associativity of wedge product, we can pair any two adjacent vector factors. Without loss of generalization, we group $y$ with the element on its left and anti-commute them, flipping the sign of the product one time. Suppose the number of factors between $x$ and $y$ is $m$. We "shift" $y$ to left $m+1$ times, to the adjacent left of $x$. Now we shall "shift" $x$ to the initial position of $y$ and this will result in $m$ changes of sign. In total $2m+1$ sign changes took place, which equals to one sign change. 
\end{proof}

A cautious eye will recognize that we are dealing with permutations here, and in fact, interchanging the order of two vector factors is equivalent to action of an odd permutation, equivalent to $2m+1$ transpositions.

\begin{definition}
    A permutation $\sigma$ is always a composition of $m$ transpositions. If this number $m$ is an odd(even) integer, we call $\sigma$ an odd(even) permutation. 

    Sign of a permutation is defined as $(sgn \;\sigma) \; := (-1)^m = \begin{cases}
        -1 & \sigma \text{ is odd}\\
        +1 & \sigma \text{ is even} 
    \end{cases}$
\end{definition}

\begin{proposition}
    Let $x_I = v_{i_1} \wedge v_{i_2} \wedge \cdots \wedge v_{i_k}$ where $v_{i_j} \in V$. Let $\sigma \in S_k$. The action of $\sigma$ on $x_I$ is given by:
    $$\sigma \cdot x_I = (sgn\;\sigma)\;x_{\sigma(I)} = (sgn\;\sigma)\; v_{i_{\sigma(1)}} \wedge v_{i_{\sigma(2)}} \wedge \cdots \wedge v_{i_{\sigma(k)}}$$
\end{proposition}

\begin{proposition}\label{prop:parity}
    Suppose $x \in A_k, y \in A_l$ are multivectors of arbitrary grades $k$ and $l$ respectively. Then, $$x \wedge y = (-1)^{|x||y|}y \wedge x$$
    Here, $|x| = k$ denotes the grade of element $x$.
\end{proposition}
\begin{proof}
    In our setup, for example, element $x$ is a homogeneous sum of elements of grade $k$ (no elements of another grade). This is also the case for y. Take any term of the product of $x$ and $y$. This term consists of $k+l$ vector factors. Shifting of first vector product of $y$ to the left of $x$ corresponds to $k$ transpositions. We apply this operation to all factors of $y$ which amounts to $l$ many times. In the end, we have made a permutation $\sigma$ act on the product $x \wedge y$. This permutation is a composition of $l$ many individual permutations which shift elements left to the $x$, each of which are combination of $k$ transpositions. In total, $\sigma$ is a composition of $k \cdot l$ permutations, which gives $(sgn \;\sigma) = (-1)^{k \cdot l} = (-1)^{|x||y|}$.
\end{proof}

\begin{proposition}\label{prop:multinil}
    Let $x, y$ be arbitrary grade irreducible elements. $$x \wedge y = 0$$ if decompositions of $x$ and $y$ both include a copy of the same basis vector $e_i$.
\end{proposition}
\begin{proof}
    The proof is similar to proof of \ref{prop:nil}. In the product $x \wedge y$, shift the copy of $e_i$ in $y$ next to the other copy in $x$. Isolate them with associativity, and from alternating property, this is 0. Hence the product $x \wedge y = 0$
\end{proof}
\begin{remark}
    In particular, from \ref{prop:multinil}, it is obvious that there are no elements of grade greater than $n$, which is the dimension of underlying vector space $V$. Also, if $x \in A_k$ and $y \in A_l$, $x \wedge y = 0$. This is a trivial result of the pigeon-hole principle from combinatorics, which dictates that one vector must be seen at least two times in the decomposition of the product $x \wedge y$
\end{remark}

\subsection{Interpretation on Euclidean Geometry}
One of the fundamental problems in geometry since ancient times until recent past was to relate numbers to geometry. Grassmann wanted to develop an algebra that describes linear (Euclidean) geometry, that extends to higher dimensions. Fearnley-Sanders, in \cite{desmond}, describes what let Grassmann to the theory of Exterior algebra, which models geometry algebraically. He states that in Ausdehnungslehre, Grassmann indicates that his initial motivation arose when wandering on the mechanisms of the formula below, that gives the relationship between lengths induced among three collinear points $A, B, C$:
$$AB + BC = AC$$
Describing:\\
\begin{minipage}{\textwidth}
\centering
\begin{tikzpicture}
    \centering
    \node[]  (a)        {$A$};
    \node[]  (b) [right=of a]    {$B$};
    \node[]  (c) [right=of b]    {$C$};
    
    \draw[-, thick]   (a.east)  to node[left] {}    node[right] {}   (b.west);
    \draw[-, thick]   (b.east)  to node[left] {}   node[right] {}   (c.west);
\end{tikzpicture}
\end{minipage}

However he realized the fascinating fact that the formula holds in any configuration of the three collinear points, if we decide on $BA = -AB$.
\\\\
Now, consider the switched locations of points $B$ and $C$ so $C$ lies in between:\\
\begin{minipage}{\textwidth}
\centering
\begin{tikzpicture}
    \centering
    \node[]  (a)        {$A$};
    \node[]  (b) [right=of a]    {$C$};
    \node[]  (c) [right=of b]    {$B$};
    
    \draw[-, thick]   (a.east)  to node[left] {}    node[right] {}   (b.west);
    \draw[-, thick]   (b.east)  to node[left] {}   node[right] {}   (c.west);
\end{tikzpicture}
\end{minipage}
According to our setting, we have:
$$AB = AC + CB = AC - BC \implies AB + BC = AC $$

Grassmann has spent considerable portion of his life to investigate how this phenomenon manifests in geometry, and eventually concluded that anti-symmetry (anti-commutativity) should be set as the defining property of an algebra that is meant to encode geometry.

In Grassmann's system, a geometric object of every dimension can be expressed algebraically. Since we strip geometric objects of their coordinate, so drop the position data, they simply exist as directions. A vector $x \in A_1 = V$ is geometrically interpreted as representing an independent direction in space (figure \ref{fig:vector}). Consider the vectors ($1$-elements) $x, y, z \in A_1$. If the space has a metric, magnitude of $x$ is interpreted as its length. A bivector $x \wedge y$, which is a linear combination of the basis of $A_2$, represents a planar direction in space (figure \ref{fig:plane}). 
\begin{definition}
\textbf{Codimension} of a $k$-vector $x \in A_k$, where $dim(V) = n$ equals to $n-k$    
\end{definition}
For visual intuition, a $2$-vector in $\mathbb{R}^3$ it can be identified with the direction of a perpendicular normal to it, which is the standard method to identify planes in $\mathbb{R}^3$. 

\begin{figure}[ht]
    \centering
    \begin{minipage}{0.4\textwidth}
        \centering
        \includegraphics[width=0.5\linewidth]{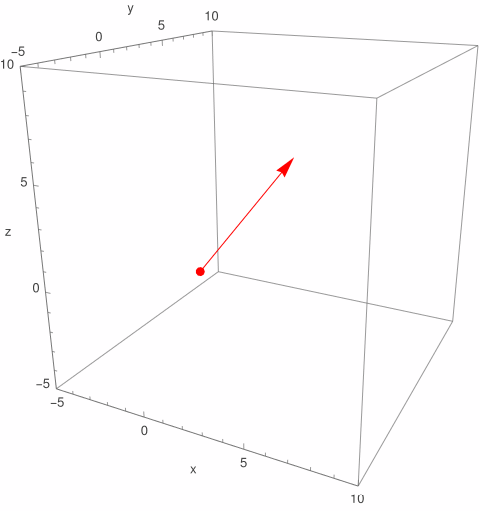}
        \caption{A direction in $\mathbb{R}^3$}
        \label{fig:vector}
    \end{minipage}
    \begin{minipage}{0.4\textwidth}
        \centering
        \includegraphics[width=0.5\linewidth]{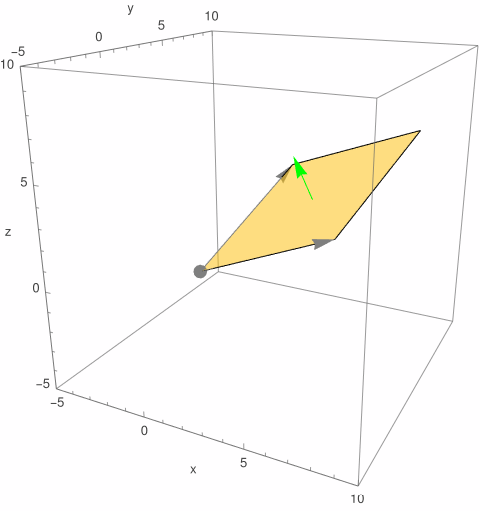}
        \caption{A planar direction in $\mathbb{R}^3$}
        \label{fig:plane}
    \end{minipage}
\end{figure}

In (figure \ref{fig:plane}), the planar dimension is identified with a vector normal to it, but note that we were able to represent it with a vector only because we are in $\mathbb{R}^3$, and \emph{codimension} of a $2$ dimensional plane is $1$ and a \emph{perpendicular space} consists of $1$-vectors. We do not dwell much on complement in exterior algebra, and the notion of codimension here.

In this case, its magnitude is interpreted as area. A trivector $x \wedge y \wedge z$, which is a linear combination of the basis of $A_3$, represents a $3$-dimensional direction in space (figure \ref{fig:plane}) (all trivectors represent the same direction in $\mathbb{R}^3$ ($dim(A_3) = 1$)) and the magnitude of a trivector represents volume. 

\begin{figure}[h]
    \centering
    \begin{minipage}{0.4\textwidth}
        \centering
        \includegraphics[width=0.5\linewidth]{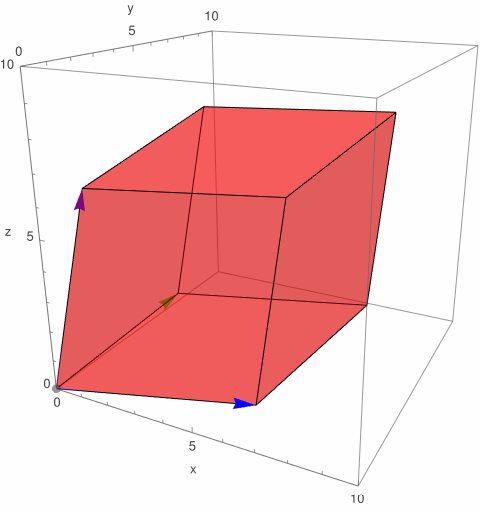}
        \caption{A trivector direction in $\mathbb{R}^3$}
        \label{fig:trivector}
    \end{minipage}
\end{figure}

Analogously, we establish the notion of volume in arbitrary dimensions.
We shall give a working example to demonstrate computation with wedge product.
\begin{example}
    Set our environment to be the three dimensional euclidean ($\mathbb{R}^3$) space with basis $\{e_1, e_2, e_3\}$. Let 
    \begin{gather*}
        x = (x_1, x_2, x_3) = x_1\,e_1 + x_2\,e_2 + x_3\,e_3\\
        y = (y_1, y_2, y_3) = y_1\,e_1 + y_2\,e_2 + y_3\,e_3
    \end{gather*}
    where $x_i, y_i$ are scalars.

    Let us give take the exterior product of $x$ and $y$ to obtain bivector $x \wedge y$ and represent it in terms of basis bivectors of $A_2$:

    \begin{align*}
    x \wedge y = \, &(x_1\,y_1)\,e_1 \wedge e_1 + (x_1\,y_2)\,e_1 \wedge e_2 + (x_1\,y_3)\,e_1 \wedge e_3 + \\ &(x_2\,y_1)\,e_2 \wedge e_1 + 
    (x_2\,y_2)\,e_2 \wedge e_2 + 
    (x_2\,y_3)\,e_2 \wedge e_3 + \\
    &(x_3\,y_1)\,e_3 \wedge e_1 + 
    (x_3\,y_2)\,e_3 \wedge e_2 + 
    (x_3\,y_3)\,e_3 \wedge e_3    
    \end{align*}

    Bivectors that are products of two copies of same basis vector $e_i$ are $0$ so they cancel out. Also, because of anti-commutativity, $e_i \wedge e_j = -e_j \wedge e_i$ so bivectors sharing the same vector factors are essentially the same, so further simplification can be made. We rewrite the product cleaner;

    $$x \wedge y = (x_1\,y_2 - x_2\,y_1)\,e_1 \wedge e_2 + (x_1\,y_3 - x_3\,y_1)\,e_1 \wedge e_3 + (x_2\,y_3 - x_3\,y_2)\,e_2 \wedge e_3$$

    One may recognize $x \wedge y$ is equal to the usual cross product of $x$ with $y$ in $\mathbb{R}^3$. Switching the bivector, for example $e_1 \wedge e_2$ with the non-apparent basis vector $e_3$ will be helpful in seeing this result. Since in $\mathbb{R}^3$, a planar direction is identical to the perpendicular direction which is the complementary dimensional direction, in this case a $3-2 = 1$ dimensional perpendicular vector.

    The cross product is only defined for $\mathbb{R}^3$ because in the product, we need to identify each linearly independent term with a direction, that is a vector. To have this property, we require codimension$ = 1$ that is only true for $\mathbb{R}^3$
    {\color{blue}}
\end{example}

\begin{example}\label{wedgedet}
    Let us perform a similar computation in $\mathbb{R}^3$ of a trivector $x \wedge y \wedge z$, for vectors
    \begin{gather*}
        x = (x_1, x_2, x_3) = x_1\,e_1 + x_2\,e_2 + x_3\,e_3\\
        y = (y_1, y_2, y_3) = y_1\,e_1 + y_2\,e_2 + y_3\,e_3\\
        z = (y_1, y_2, y_3) = z_1\,e_1 + z_2\,e_2 + z_3\,e_3
    \end{gather*}
    After a series of tedious simplifications on the product, we obtain
    $$
    x \wedge y \wedge z = (
    x_1\,y_2\,z_3 - x_1\,y_3\,z_2 - x_2\,y_1\,z_3 + x_2\,y_3\,z_1 + x_3\,y_1\,z_2 - x_3\,y_2\,yz_1 ) \, e_1 \wedge e_2 \wedge e_3
    $$
    {\color{blue}}
\end{example}

We must point out a very important observation here, that is, the coefficient is the determinant of vectors $x, y, z$. This determinant trivially equals the volume of the parallelepiped spanned by these three vectors in a metric space. This is no coincidence, and there is a direct relationship between the determinant, a central algebraic object in geometry, and exterior algebra. We postpone investigation of this relationship to Section 4.

\pagebreak

\section{Formal Construction of Exterior Algebra}

In this section, we focus on the construction of Exterior algebra from free associative algebra, namely the \emph{Tensor algebra}. The Tensor algebra may be viewed as a non-commutative generalization of the Tensor Products. The construction reveals how the structure of Tensor algebra -which is central in physics- and Exterior algebra are intrinsically connected, and hints us the great power of Exterior Algebra to accurately model physical phenomena and geometry.
To clarify this construction, we first compare this it with a similar but more familiar example: the construction of the polynomial ring in $n$ variables from free associative algebra generated by a set of variables $X$.

\subsection{Polynomial Ring}
As we stated, we will first introduce the construction of Polynomial Rings in $n$ variables as a familiar pillar to stand on while observing the construction of the tensor algebra.

\begin{definition}
Let $R$ be a commutative ring with identity, and let $X = \{x_1, x_2, ..., x_n\}$ be a set of $n$ variables. The \textbf{Polynomial Ring} in $n$ variables over $R$ is denoted $R[X] = R[x_1, x_2, ..., x_n]$
\end{definition}
\begin{remark}
    It is well known that Polynomial Ring $R[x_1, ..., x_n]$ is a commutative $R$-algebra with usual multiplication and scalar addition.
\end{remark}

Let $A$ be a polynomial ring in $n$ variables. An element $p \in A$ is a $R$-linear combination of monomials $x_1^{a_1}x_2^{a_2}...x_n^{a_n} \text{ where } a_i \in \mathbb{N}_0$. Note that $\mathbb{N}_0 = \mathbb{N} \cup \{0\}$.
So every element looks like:
$$p = \sum_{I=(a_1, ..., a_2)} c_I \; x_1^{a_1}x_2^{a_2}...x_m^{a_m} \quad \text{with} \quad c_I \in R \quad \text{ I are multi-indices}$$

\begin{proposition}
The following are the defining properties of polynomial ring:
\begin{itemize}
    \item Commutativity: the multiplication is commutative such that $x_ix_j = x_jx_i$
    \item A free \textbf{commutative} algebra: (Not to be confused with free associative algebra) The variables are algebraically independent meaning that there are no non-trivial polynomial relations among them, aside from the ones forced by commutativity and associativity.
    \item Connected graded: this algebra is naturally graded by total degree and connected, which is $$A = \bigoplus^\infty_{i = 0}A_i, \quad A_d = \text{span of degree d monomials.}$$
\end{itemize}
\end{proposition}
\begin{proof}
    Commutativity follows from the definition of $R[X]$. Algebraic independence is essentially the defining property of a polynomial ring. Connected graded structure can be trivially verified from the definition of monomial degree. These are standard properties of polynomial ring and follow directly from its construction. 
\end{proof}

With the exception of anti-commutativity - instead of commutativity, we want anti-commutativity - these items encompass all properties we wish exterior algebra to have. We will see this similarity is also shared in both constructions with minor tweaks. Let us go back to our investigation of the construction of polynomial rings.
Primarily, the algebra is built on the free associative algebra over $R$ generated by the set of variables $X$.
\newline

Our method to construct the polynomial ring will be as shown on the diagram below:

\begin{center}
\begin{minipage}{\textwidth}
\centering
\begin{tikzpicture}
    \centering
    \node[STEPBOX]  (a)        {Free associative algebra};
    \node[STEPBOX]  (b)    [below=of a]   {Polynomial Ring};
    
    \draw[->, very thick]   (a.south)  to node[right] {Impose commutativity relation by taking quotient}   (b.north);
\end{tikzpicture}
\end{minipage}
\end{center}

Consider the free associative algebra of $n$ variables over a set $X = \{x_1, \cdots, x_n\}$.
Notice that this is a primitive notion, in the sense that there are no other relations except:
\begin{itemize}
    \item \emph{associativity} that comes from grouping the adjacent symbols together
    \item \emph{$R$-linearity} which refers to the fact that each element of the algebra is a finite linear combination of formal words, with \emph{scalars} from $R$ as coefficients for the terms.
\end{itemize}
For example it is not commutative, such that $x_1x_2 \neq x_2x_1$ even if $x_1 = x_2$ in value. The key to understanding this structure lies in regarding each word of symbols as a distinct element, regardless of the values of the symbols.

To obtain a commutative polynomial ring from this structure, we must impose a commutativity relation on this crude algebraic structure, thereby forcing the variables to commute. This is achieved by taking the quotient of the free algebra with the two-sided ideal generated by the commutators $\langle x_ix_j - x_jx_i \rangle$:
$$R[x_1, \dots, x_n] := R\langle x_1, \dots, x_n \rangle \big/ \langle x_ix_j - x_jx_i \rangle$$

By doing so, we have rendered $x_ix_j - x_jx_i = 0 \Longleftrightarrow x_ix_j = x_jx_i$ (equality of cosets) in our definition of the polynomial ring, which effectively guarantees that the commutativity relation holds. 

\begin{example}
Let's give an example to show how commutativity imposed this way carries to monomials of greater degrees and thus is global over the algebra:

\begin{equation*}
    x_1^2x_3 = x_1x_1x_3 = x_1(x_1x_3) \underset{\mathrm{commutativity}}{=} x_1(x_3x_1) = (x_1x_3)x_1 \underset{\mathrm{commutativity}}{=} (x_3x_1)x_1 = x_3x_1x_1 = x_3x_1^2
\end{equation*}
\end{example}

\subsection{Exterior Algebra}
Contrary to the Polynomial Ring construction, we do not start from the same free associative algebra. Instead we begin with a more primitive notion -the free vector space on a set- and develop tensor product by imposing bilinearity relations on it. Our goal in doing so is to reveal the bilinear nature of tensor products, a feature that will also be present in \emph{wedge (exterior) product}. We then proceed to generalize this construction to the tensor algebra, and ultimately impose anti-commutativity to obtain exterior algebra.

\begin{center}
\begin{minipage}{\textwidth}
\centering
\begin{tikzpicture}
    \centering
    \node[STEPBOX]  (a)        {(Step 1) Free vector space over a set ($V \times V$)};
    \node[STEPBOX]  (b) [below=of a]    {Tensor Product $V \otimes V$};
    \node[STEPBOX]  (c) [below=of b]    {Tensor Algebra (free associative algebra)};
    \node[STEPBOX]  (d) [below=of c]    {Exterior Algebra};
    
    \draw[->, very thick]   (a.south)  to node[left] {Step 2, 3}    node[right] {Bilinearity relations (Quotient)}   (b.north);
    \draw[->, very thick]   (b.south)  to node[left] {Step 4}   node[right] {Algebra structure by graded sum}   (c.north);
    \draw[->, very thick]   (c.south)  to node[left] {Step 5}   node[right] {Anti-commutativity relation (Quotient)}   (d.north);
\end{tikzpicture}
\end{minipage}
\end{center}

We first define the concept of \emph{free vector space over a set}.

\begin{definition}\textbf{Free vector space on a set} (definition 1)
    Let S be a set. The free vector space over set S, F(S) is a vector space that allows arbitrary linear combinations $\sum a_i s$ of all formal words $s \in S$ as in the previous case of free non-commutative algebra. In set notation,
    $$F(S) := \left\{\sum_{s \in S} c_s s \middle | s \in S, c_s \in \kk\right\}$$
\end{definition}
\begin{definition}
    Let $S$ be any set. The \textbf{support} for a map $\varphi : S \to \kk$ is defined as $$supp(\varphi) := \{ s \in S \mid \varphi(s) \neq 0 \}$$
    i.e, the support of $\varphi$ is the set of points in $S$ where $\varphi$ does not vanish.
\end{definition}
\begin{definition} \textbf{Free vector space on a set} (definition 2) 
    Let $S$ be any set. The subset $\mathcal{F}(S)$ of $Fun(S, \kk)$ is defined as follows:
    $$\mathcal{F}(S) = \{\varphi \in Fun(S, \kk) \mid \varphi(s) \neq 0 \text{ for finitely many } s \in S \}$$
    The mappings that constitute $\mathcal{F(S)}$ are also known as linear mappings from $S$ to $\kk$ with finite support.

    We know that $Fun(S, \kk)$ is a vector space. $\mathcal{F(S)}$ defined as above is a subspace of $Fun(S, \kk)$ and called the free vector space over the set $S$.
\end{definition}
\begin{remark}
    The equivalence of the two given definitions of a free vector space on a set $S$ is a subtle point. To clarify it, one must consider the map that takes a linear combination $\sum_{s \in S}c_s s$ of formal words in $S$ to a function $\varphi : S \to F$ such that $\varphi(s) = c_s$ This way, it is clear that an $f \in \mathcal{F(S)}$ encodes a linear combination, by mapping a formal word (element) in $S$ to its coefficient in the linear combination.
\end{remark}

\subsubsection{Tensor Product}
Let us begin with the first step of construction of the tensor product.

\paragraph{Step 1}
Let $V, W$ be vector spaces. The free vector space over the cartesian product of these spaces is $F(V \times W)$. In this way, we have introduced sum and scalar multiple of elements to our space.
One point to pay attention is that $F(V \times W)$ is infinite dimensional, because there is no relation of "being scalar multiple of each other" in a free vector space (not free algebra!), since the space is "free" of all relations except $\kk$-linearity. Therefore, for example $2(v, w) \neq (2v, 2w)$ and both $(v, w)$ and $(2v, 2w)$ are distinct basis vectors in $F(V \times W)$. Consequently, we have infinitely many linearly independent elements, each forming a basis element, hence the infinite dimension.

\paragraph{Step 2}
We decide on the relations to be imposed. In our case, we wish to make vector addition and scalar multiplication bilinear, which requires the 4 relations given below:

\begin{itemize}
    \item $(v + v', w) - (v, w) - (v', w)$
    \item $(v, w + w') - (v, w) - (v, w')$
    \item $(cv, w) - c(v, w)$
    \item $(v, cw) - c(v, w)$
\end{itemize}

We will see how these impose bilinearity on $F(V \times W)$ in the next step.

\begin{remark}
    There is a distinction between a free vector space and a free associative algebra that is worth highlighting. Unlike a free vector space, in a free asociative algebra we have $$(x_i+x_j)x_k = x_ix_k + x_jx_k$$ in other words, bilinearity of addition and scalar multiplication already hold because of the extra multiplication that belongs to the structure. The algebra satisfies these relations as identities, originating from the extended linearity of the multiplication.
\end{remark}

\paragraph{Step 3}
First we, create equivalent classes of elements that are equal according to bilinearity by forming the generated by all 4 relations:

\begin{align*}
    R := \langle\, &(v + v', w) - (v, w) - (v', w), \\
    &(v, w + w') - (v, w) - (v, w'), \\
    &(cv, w) - c(v, w), \\
    &(v, cw) - c(v, w) \, \rangle \quad \text{where} \quad v, v' \in V, w, w' \in W, c \in \kk
\end{align*}

Next impose the relations of previous step by taking the quotient of $F(V \times W)/R$ by the subspace $R$ generated by 4 expressions above. This yields the quotient named the tensor product of vector spaces:

$$ V \otimes W :=  F(V \times W)/R $$

\begin{definition}
    We define the \textbf{tensor product} $v \otimes w := [(v, w)]$ of elements $v \in V, w \in W$ to be the equivalence class $[(v, w)]$ of $(v, w)$ in $F(V \times W)/R$
\end{definition}

\begin{proposition}
    The natural map 
    \begin{gather*}
        \varphi: V \times W \longrightarrow V \otimes W \\
        (v, w) \mapsto v \otimes w
    \end{gather*}
    is bilinear.
\end{proposition}
\begin{proof}
    We show how our statement holds for one relation in detail and the other three will follow the same process:
    \begin{gather*}
        (v + v', w) - (v, w) - (v', w) \in R\\
        \text{ thus } \; (v + v', w) - (v, w) - (v', w) + R = 0 + R \; \text{ in } \; V \otimes W \\
        \text{it follows that } (v + v', w) + R =  (v, w) + (v', w) + R
    \end{gather*}
    We know 
    \begin{gather*}
        (v+v')\otimes w = [(v + v', w)] = (v + v', w) + R\\
        \text{and }\\ 
        v \otimes w + v' \otimes w = [(v, w)] + [(v', w)] = (v, w) + (v', w) + R
    \end{gather*}
    Then, 
    \begin{equation*}
        \varphi(v+v', w) = (v+v')\otimes w = v \otimes w + v' \otimes w = \varphi(v, w) + \varphi(v', w)
    \end{equation*}

    The remaining cases follow by similar arguments:
    \begin{itemize}
        \item $v \otimes (w, w') = v \otimes w + v \otimes w'$
        \item $ (cv) \otimes w = c(v \otimes w)$
        \item $ v \otimes (cw) = c(v \otimes w)$
    \end{itemize}

    Hence we have shown bilinearity is inherited by the given four relations.
\end{proof}
\subsubsection{Transition to Exterior Algebra}
\paragraph{Step 4}
Now we have at our disposal the tensor product, 
\begin{gather*}
    \otimes : V \times W \longrightarrow V \otimes W \\
    (v, w) \mapsto v \otimes w
\end{gather*}
This is a bilinear operation on top of the free vector space structure.
Using tensor products of vector spaces, we then construct the tensor algebra, which serves as the analogue of the free associative algebra $k\langle x_1, ..., x_n\rangle$. Both of these objects are examples of free associative algebras. 
Afterwards, in contrast to the polynomial ring case, we will impose an "anti-commutativity" relation on the algebra to finally obtain the exterior Algebra that we seek after.

\begin{definition}
    We define the tensor algebra $T(V)$ on V as the direct sum:
    $$T(V) = \\(k) \oplus (V) \oplus (V \otimes V) + \cdots = \bigoplus_{n = 0}^\infty V^{\otimes n} $$
\end{definition}

As mentioned before, the tensor algebra is a free associative algebra on V, such that we do not yet have a commutativity or anti-commutativity relation imposed upon. Now we introduce the anti-commutativity to $T(V)$. There are two ways to achieve this. We shall explain each method with their advantages and drawbacks.

\paragraph{Method 1: Using relation $v \otimes v = 0$}\mbox{}\\
This method orders us to take the quotient of $T(V)$ with $\langle v \otimes v\rangle$, the subspace generated by $v \otimes v$. 
\begin{notation}
    $\langle v \otimes v\rangle$ = $\langle v \otimes v \mid v \in V \rangle$.
\end{notation}
Let $E(V) = T(V)/\langle v \otimes v\rangle$.
\begin{notation}
    From now on denote the multiplication with the wedge symbol $\wedge$, instead of $\otimes$.   
\end{notation}
How does anti-commmutativity reveals itself in this setup? It is a slightly subtle point.

\begin{proposition}
    Multiplication on $E(V)$ is anti-commutative.
\end{proposition}
\begin{proof}
    \begin{align*}
        0 = &(v+w) \wedge (v+w)\\
        = &(v \wedge v) + (v \wedge w) + (w \wedge v) +  (w \wedge w) \quad &(\text{distributivity})\\
        = &0 + (v \wedge w) + (w \wedge v) + 0 \quad &(\text{in } E(V))\\
        = &(v \wedge w) + (w \wedge v)
    \end{align*}
    Thus we have
    $$(v \wedge w) = -(w \wedge v)$$
\end{proof}

It seems we have successfully implemented anti-commutativity to our multiplication operation. However, there is a caveat: what if we are working on a field of characteristic 2? Since in a field $F$ of characteristic 2 for an element $a \in F$, we have $(-1)a = -a = a$. Therefore our final equation becomes:
$$(v \wedge w) = -(w \wedge v) = (w \wedge v)$$
This means instead of anti-commutativity, we get commutativity! This happens because by taking quotient with $\langle v \otimes v\rangle$ we do not directly create equivalence classes of anti-commutative elements, rather we impose it implicitly, therefore our control over the conditions is limited.

\paragraph{Method 2: Using relation $v \otimes w + w \otimes v = 0 $}\mbox{}\\
This time, when we take the quotient $E(V) = T(V)/\langle v \otimes v\rangle$, we directly create equivalence classes of anti-commutative elements, thus rendering $v \wedge w = -w \wedge v$. The reader is encouraged to check that the construction of exterior Algebra $E(V)$ with this method gives us the desired multilinear, anti-commutative algebra over any vector space V, including the ones defined over fields of characteristic 2.

Here is an overview of the construction of polynomial ring in $n$ variables and exterior Algebra together for comparison.

\vspace{1em}
\begin{minipage}{0.3\textwidth}
\centering
\vspace{3.5cm}
\begin{tikzpicture}
 \centering
    \node[STEPBOX]  (a)        {Free associative algebra};
    \node[STEPBOX]  (b)    [below=of a]   {Polynomial Ring};
    
    \draw[->, very thick]   (a.south)  to node[right] {Commutativity}   (b.north);
\end{tikzpicture}
\end{minipage}
\hfill
\begin{minipage}{0.55\textwidth}
\centering
\begin{tikzpicture}
\centering
    \node[STEPBOX]  (a)        {(Step 1) Free vector space over a set ($V \times V$)};
    \node[STEPBOX]  (b) [below=of a]    {Tensor Product $V \otimes V$};
    \node[STEPBOX]  (c) [below=of b]    {Tensor Algebra (free associative algebra)};
    \node[STEPBOX]  (d) [below=of c]    {Exterior Algebra};
    
    \draw[->, very thick]   (a.south)  to node[left] {Step 2, 3}    node[right] {Bilinearity relations}   (b.north);
    \draw[->, very thick]   (b.south)  to node[left] {Step 4}   node[right] {Algebra structure}   (c.north);
    \draw[->, very thick]   (c.south)  to node[left] {Step 5}   node[right] {Anti-commutativity relation}   (d.north);
\end{tikzpicture}
\end{minipage}
\vspace{1em}

\pagebreak

\section{The Determinant}
\subsection{The Universal Property}
In this section letter $A$ does not specifically denote the Exterior algebra over a vector space $V$, but any suitable algebra determined by the context. Instead, we denote the Exterior algebra as $E(V)$ 
The construction given in detail in the preceding section is not merely a quotient algebra, but it is the best ("freest") algebra satisfying anti-commutativity. 

\begin{remark}
The notion of universality is originally a concept from category theory. As with Remark \ref{free}, we will not refer to the original categorical definition. Instead, we content ourselves with an algebraic version of it.
\end{remark}

A universal property defines an object by how every other object of similar type can be uniquely reached from it via structure-preserving maps (homomorphisms).

\begin{definition} Universal Property \\
Let V be a vector space, and let $\varphi : V \to A$ be a canonical map into some algebra A, typically the quotient $A = T(V)/I$ for some suitable ideal $I$. Let $f : V \to B$. be a linear map into another algebra $B$ satisfying certain compatibility conditions (e.g., anti-commutativity) We say that A satisfies a universal property if there exists a unique homomorphism $\bar{f} : A \to B$ such that $f = \bar{f} \circ \varphi$. In other words, there exists a unique map $\bar{f}$ that extends $f$ such that the diagram below commutes:

$$
\begin{tikzcd}[row sep = 4em, column sep = 4em]
& V \arrow{r}{\varphi} \arrow{dr}{f} & A \arrow{d}{\bar{f}} \\ & & B
\end{tikzcd}
$$
\end{definition}

In particular, Exterior algebra satisfies the universal property. A unital algebra is an algebra with multiplicative identity. Let $A$ be any unital algebra and $f$ a linear map such that following conditions hold:

\begin{itemize}
    \item $f(v)f(v) = 0 \quad \forall v \in V$
    \item $f(v)f(w) + f(w)f(v) = 0 \quad \forall v, w \in V$
\end{itemize}
Without proof, we give this fact:

\begin{proposition}
    There exists a unique algebra homomorphism $\bar{f} : E(V) \longrightarrow A$ extending $f$ such that $f = \bar{f} \circ \varphi$, where $\varphi : V \longrightarrow E(V)$ is the canonical map $v \mapsto v \text{ mod } I$, where $I = \langle x_ix_j + x_jx_i \rangle$, i.e. the ideal that encodes anti-commutativity relation.
\end{proposition}

Therefore, the following diagram commutes:

$$
\begin{tikzcd}[row sep = 4em, column sep = 4em]
& V \arrow{r}{\varphi} \arrow{dr}{f} & E(V) \arrow{d}{\bar{f}} \\ & & A
\end{tikzcd}
$$

In particular $E(V)$ is initial among pairs $(A,f)$ consisting of a unital $\kk$–algebra $A$ and a linear map $f:V\to A$ with $f(v)^2=0$ for all $v\in V$, which means $E(V)$ is the \emph{freest} algebra in which products of equal generators vanish (and hence, when $2$ is invertible ($char(\kk) \neq 2$), the generators anti-commute). 

Put another way, any alternating linear map from $V$ to an algebra $A$ extends uniquely to an algebra homomorphism from the exterior algebra $E(V)$.
This is what makes exterior algebra "universal" for enforcing anti-symmetry.

\subsection{Uniqueness of the Determinant}
For the following discussion, let $A$ be a $n \times n$ matrix where $n \in \mathbb{N}$ and $V$ be an $n$-dimensional vector space over field $\kk$. $a_{i, j}$ denotes the entry of $A$ on the row $i$ and column $j$.

\begin{definition}
    \textbf{ij - Minor} $M_{ij}$ is the $(n-1) \times (n-1)$ matrix obtained by deleting row $i$ and column $j$ from $A$.
    the \textbf{Cofactor} $C_{ij}$ is defined as $(-1)^{i+j}\left| M_{ij}\right|$, where $\left| M_{ij}\right|$ is the determinant of the minor $M_{ij}$.
\end{definition}

\begin{definition}
    \textbf{Cofactor expansion}, or \textbf{Laplace Expansion}, is a sum defined on a $n \times n$ matrix for a chosen row or column. Without loss of generalization we give the definition of cofactor expansion for $i^th$ row of a $n \times n$ matrix $A$.

    $$S_i := \sum_{j = 1}^n a_{ij}\,C{ij}$$
\end{definition}
Trivially, we know that cofactor expansion of all rows and columns are equal in value.
Also, we know that Leibniz formula that runs over permutations of entries of $A$, given as follows:
$$\sum_{\sigma \in S_n}(sgn \,\sigma)\,\prod_{i = 0}^na_{i, \sigma(j)}$$
equals to the determinant of matrix $A$, $det(A)$.

\begin{theorem}
    Leibniz formula and Cofactor expansion are equivalent:
    $$det(A) = \sum_{\sigma \in S_n}(sgn \,\sigma)\,\prod_{i = 0}^na_{i, \sigma(j)} = S_i = \sum_{j = 1}^n a_{ij}\,C{ij}$$ 
\end{theorem}
\begin{proof}
    We prove the equality of formulas by induction on $n$. Also, we consider, without loss of generalization, $1^{st}$ row cofactor expansion $S_1$.
    
    Base case: $n = 1$
    We consider $A = [a_{11}]$
    Sign of the identity permutation is $+1$ and cofactor in this case is defined to be $1$. Thus base case holds.

    Inductive hypothesis:
    Assume for all $(n-1)\times(n-1)$ matrices, the Leibniz formula equals cofactor expansion, i.e. the determinant.

    Let $\sigma \in S_n$ and $\bar{\sigma}$ be the restriction of $\sigma$ to $\{2, ..., n\}$. $\sigma(1)$ cannot be in the image of $\bar{\sigma}$ and since there is $n-1$ elements in the image and we have $n$ possible slots, by pigeonhole principle we have one option, which we denote as $\sigma(1) := j_0$ 
    Then for the particular extension $\sigma$, we have:
    $$
    \prod_{i = 1}^na_{i\sigma(i)} = a_{1j_0}\prod_{i = 2}^n a_{i\bar{\sigma}(i)}
    $$

    Moving $j_0$ to the first position, such that $\sigma(1) = j_0$ contributes $(-1)^{1+j_0}$ to the sign relative to $\bar{\sigma}$, the permutation of the remaining indices. So we relate signs of $\sigma$ and $\bar{\sigma}$ as follows:
    $$(sgn\;\sigma) = (-1)^{1+j_0}(sgn\;\bar{\sigma})$$ 

    By inductive hypothesis, the determinant of the minor $M_{1j_0}$, which is a $(n - 1) \times (n - 1)$ matrix is:
    $$det(M_{1j_0}) = \sum_{\bar{\sigma} \in S_{n-1}}(sgn\;\bar{\sigma})\prod_{i = 2}^n a_{i \bar{\sigma}(i)}$$

    For each $\sigma$ with $\sigma(1) = j_0$:
    $$
    (sgn\;\sigma)\prod_{i=1}^n a_{i\sigma(i)} = (-1)^{1+j_0}(sgn\;\bar{\sigma}) a_{1j_0} \prod_{i = 2}^n a_{i \bar{\sigma}(i)} = a_{1j_0} (-1)^{1+j_0}det(M_{1j_0})
    $$

    Now we partition $S_n$ according to $\sigma(1) = j$ for $j \in \{1, ..., n\}$. Then summing over all $\sigma$:

    \begin{align*}
        det(A) &= \sum_{\sigma \in S_n}(sgn\;\sigma)\prod_{i = 1}^n a_{i\sigma(i)}\\
        &= \sum_{j = 1}^n a_{1j}(-1)^{1+j} \sum_{\bar{\sigma} \in S_{n-1}} \prod_{i = 2}^n a_{i\bar{\sigma}(i)}\\
        &= \sum_{j = 1}^n a_{1j}(-1)^{1+j}\,det(A_{1j}) = \sum_{j = 1}^n a_{1j}\,C_{1j} 
    \end{align*}
\end{proof}

The determinant has several defining properties that are listed below.
\begin{corollary}\label{cor:detprop}
    Determinant $det : V^n \to \kk$ is an alternating function that is multilinear in rows, maps standard basis of $V$ to $1$.
\end{corollary}
\begin{proof}
    The fact that the determinant of the identity matrix equals 1 is immediate from the definition. The rest of the properties follow easily from the Leibniz formula. 
\end{proof}

We desire to show that determinant is the only function from $V^n$ to $\kk$ that satisfies the properties given in \ref{cor:detprop}. Recall that, according to the universal property of exterior algebra, any alternating linear $f: V \to A$ is factors into two steps, and goes through exterior algebra $E(V)$ of $V$, first being the canonical map to $E(V)$ and second is the unique homomorphism from it.

\begin{theorem}\label{detunique}
    Determinant is the unique function that satisfies the properties given in \ref{cor:detprop}
\end{theorem}
\begin{proof}
    Consider the diagram below:
    $$
    \begin{tikzcd}[row sep = 4em, column sep = 4em]
    & V^n \arrow{r}{\varphi} \arrow{dr}{det} & E(V) \arrow{d}{\overline{det}} \\ & & F
    \end{tikzcd}
    $$
    We have changed the setting to appropriately accommodate any alternating determinant function denoted $det$. Now we consider the vector space $V^n$ and field $\kk$ as an algebra. The canonical map is defined explicitly as
    \begin{align*}
        \varphi : V^n &\to \bigwedge^n V\\
        (v_1, \cdots, v_n) &\mapsto v_1 \wedge \cdots \wedge v_n
    \end{align*} i.e. $\varphi$ is an injection into $\bigwedge^nV \subseteq E(V)$, which is linear in each argument. \\
    We stated that there is a unique homomorphism (linear map) $\overline{det} :  \bigwedge^n V \subseteq E(V) \to \kk$ that corresponds to $det$, such that $$f(v_1, \cdots, v_n) = \bar{f}(v_1, \cdots, v_n)$$ Which induces a isomorphism $\theta : Hom(V^n, \kk) \to Hom(\bigwedge^nV, \kk)$. However $\bigwedge^nV$ is 1-dimensional. A linear transformation is uniquely by its mapping of basis elements. In this case, any non-zero vector constitutes a basis. We choose the canonical basis consisting of element $e_1 \wedge \cdots \wedge e_n$ only, and the identity matrix $I \in M_{n \times n}$. $det(I) = 1$, we have equality $\overline{det}(e_1 \wedge \cdots \wedge e_n) = 1$ which characterizes $\overline{det}$. Therefore $$\left|Hom(\bigwedge^nV, \kk)\right| = 1$$ Thus $\overline{det}$ is unique up to a scalar, hence the corresponding $det$ is the unique determinant function.
\end{proof}

Computation in \ref{wedgedet} hinted an intrinsic relationship between wedge product and determinant. This relationship and uniqueness of determinant manifests itself in the following identity:

\begin{corollary}
    $$v_1 \wedge \cdots \wedge v_n = det(A)\,e_1 \wedge \cdots \wedge e_n$$ where $A$ is the matrix $[v_1 \, \cdots \, v_n]$ where each $v_i$ occupies a column.
\end{corollary}
\begin{proof}
    We regard $e_1 \wedge \cdots \wedge e_n$ as the standard basis of  $\bigwedge^nV$. Any $n$-vector $v = v_1 \wedge \cdots \wedge v_n \in \bigwedge^nV$ can be written uniquely as a scalar multiple $v = c \cdot e_1 \wedge \cdots \wedge e_n$ for some $c \in \kk$, for $\bigwedge^nV$ is $1$-dimensional. Given identity induces a map 
    \begin{align*}
        \alpha : \bigwedge^nV &\to \kk\\
        v &\mapsto c
    \end{align*}
    $\alpha$ is obviously multilinear, alternating and we have $\alpha(e_1 \wedge \cdots \wedge e_n) = 1$. By \ref{detunique}, $\alpha$ is the determinant.
\end{proof}

\begin{remark}
    This implies that determinant is a natural consequence of the wedge product, and wedge product is the underlying structure for the determinant.
\end{remark}

\subsubsection{Volume}
Suppose we have a metric on the vector space $V$, therefore the notion of volume is defined. A volume is obtained by multiplying $n$ linearly independent vectors, which prompts us again to consider the $\bigwedge^nV$. We wish to give an intuitive explanation for the relationship between volume and determinant, since we do not care about a rigorous setup of metric on a space and a careful definition of volume in this paper. For visualization, one may consider our explanations on euclidean space $\mathbb{R}^3$. Say we have three vectors, namely $v_1, \cdots, v_n \in V$. We can talk about the parallelotope spanned by these vectors, denote its oriented volume as $vol(v_1, \cdots, v_n)$. Suppose we change one of the vectors, namely $v_i$ with $v_i' = c\,v_1 + x$, where choice of $c \in \kk$ and $x \in V$ are arbitrary. Intuitively, we have $vol(v_1, \cdots, v_i', \cdots v_n) = c \cdot vol(v_1, \cdots, v_i, \cdots v_n) + vol(v_1, \cdots, x, \cdots v_n)$ This should hint at multilinearity of volume. Recall from linear algebra that volume is oriented and its orientation changes when we switch order of any two vector factors, which is linked to sign change in wedge product. Finally, volume of a product of the standart basis, namely $vol(e_1, \cdots, e_n) = 1$. Therefore volume we identify volume with the unique function that shares this properties, the determinant. We summarize this relationship symbolically as $vol = det$.

\pagebreak

\section{Invariant Subspaces and Subalgebras}

It is shown in \cite{zahra} that a (not necessarily commutative) $\kk$-algebra $R$ is a polynomial ring over $\kk$ if and only if it is finitely generated, connected graded with no non-trivial invariant subspaces. Recall that a $\kk$-subspace $W$ of $R$ is called invariant (or Aut-stable) if $f(W)\subseteq W$ for every automorphism $f \in \text{Aut}_{\kk}(R)$. Recall that elements of $\text{Aut}_{\kk}(R)$ are ring isomorphisms which are also ${\kk}$-linear. The subspaces $0$, $\kk$, and $R$ are called trivial subspaces. Although there are automorphisms that respect (preserve) grading such that $f(A_i) \subseteq A_i$, here we consider all automorphisms; regardless if they respect gradings, or not. First, we summarize the classification of automorphisms of exterior algebra given in \cite{djokovic}.

\subsection{Automorphisms}

For the Grassmann algebra generated by $n$ variables, when $n$ is finite, the automorphisms are completely understood and studied in the groundbreaking work of D. {\v{Z}}. Djokovi{\'c}.

We first give an abstract description of automorphism group of exterior algebra $A$. To do that, we need to develop some algebraic notions. For the discussion below, we consider $dim(V) = n$ finite and $char(\kk) \neq 2$. Let $\Gamma$ be the group of automorphisms of $A$.

\begin{definition}
    An algebra is called a \textbf{local algebra} if it has only one maximal ideal.
\end{definition}

\begin{proposition}\label{local}
    A is a local algebra with the unique maximal ideal $M = \sum_{i=1}^n A_i$
\end{proposition}
\begin{proof}
    Trivially, have that $A/M \cong A_0 = \kk$ as a result of the first isomorphism theorem. It is well known that since $A/M$ is a field, there is no ideal $J$ of $A$ such that $M \subset J \subset R$, so $M$ is maximal. This is the unique maximal ideal, since in any ring every maximal ideal equals the set of non-units in the ring. Equivalently, every maximal ideal is contained in the complement of the set of units, which in this case is $\kk$. Thus the set of non-units is precisely $M$, and every maximal ideal must be $M$. Thus it is the unique maximal ideal of $A$.
\end{proof}

\begin{notation}
    We define ideals $$M_k := \sum_{i = k}^n A_i$$ Note that $M^0 = A$.
\end{notation}

This notation is consistent with the algebra structure, since the multiplication respects grading. Indeed, for all $i$, we have $A_i \wedge A_i \subseteq A_{2i}$. Thus the product of elements from $M^k$ lies again in $M^{2k} \subseteq M^k$, so the $M^k$ form a multiplicative filtration, which justifies the notation.
(A filtration is a nested subsequence of subsutructures compatible with multiplication.)

The ideals $M^i$ form a descending chain:
$$A = M^0 \supset M^1 \supset \cdots M^n \supset 0$$

\begin{definition}
    We say that an automorphism $\sigma$ \textbf{stabilizes} a chain $S_0 \subset S_1 \subset S_2 \subset \cdots$,\\ if $\sigma(S_i) = S_i$ for $0 \le i \le n$.
\end{definition}
It follows that inclusion relation remains true under a chain stabilizing automorphism.

\begin{proposition}    
    Any automorphism $\sigma \in \Gamma$ stabilizes the finite chain of ideals $M_i$ given above. 
\end{proposition}
\begin{proof}
    Since $M$ is the unique maximal ideal of $A$ by \ref{local}, an automorphism $\sigma \in \Gamma$ must map it to itself, i.e. $\sigma(M) = M$. It follows that $$\sigma(M^k) = \sigma(M)^k = M^k$$ 
\end{proof}

\begin{definition}
    Let $\sigma \in \Gamma$. For some $\;0 \le i \le n$, we have $\sigma(M^i) = M^i$ and $\sigma(M^{i+1}) = M^{i+1}$. Then $\sigma$ naturally induces a linear map $\sigma_i$ on the quotient, defined as:
    \begin{align*}
        \sigma_i : M^i/M^{i+1} &\to M^i/M^{i+1}\\
        x + M^{i+1} &\mapsto \sigma(x) + M^{i+1}
    \end{align*}    
\end{definition}

\begin{proposition}
    Let $\sigma \in \Gamma$ and $\sigma_i$ is defined as in the previous definition. $\sigma_i$ is an automorphism of vector space $M^i/M^{i+1}$
\end{proposition}
\begin{proof}
    $\sigma_i$ is well-defined, which means that the map is unambiguous on cosets. In other words, the map does not depend on the choice of representative of the coset: Say $x, y$ are in the same coset, such that $x-y \in M^{i+1}$. Then, $$\sigma_i(x) - \sigma_i(y) = \sigma_i(x-y) \in \sigma_i(M^{i+1}) = 0 + M^{i+1} $$
    $\sigma_i$ inherits linearity and bijectivity from $\sigma$.
\end{proof}

There is a canonical map $\varphi : A_i \to M^i/M^{i+1}$ so we can identify $\sigma_i$ as an automorphism of $A_i$.

Note that $A_i$ is a vector space and the set of automorphisms of $A_i$ is $GL(A_i)$. The map $f_i : \Gamma \to GL(A_i)$ defined by $f_i(\sigma) = \sigma_i$ is clearly an homomorphism. For any $\sigma \in \Gamma$, $\sigma_0$ is the identity such that $\sigma_0(c) = c\cdot\sigma_0(1) = c \cdot 1 = c$ for $c \in A_0$. It is well-known that $f_1$ is surjective. This is due to the following fact: Since $A_1 = V$ is the genarating space of algebra $A$, any linear change in a basis of $V$ determines a unique algebra automorphism of the whole exterior algebra A. We denote $GL(A_1)$ with $G$, onward. In particular, any $\tau \in G$ can be extended multiplicatively to an automorphism $g(\tau) \in \Gamma$, such that $$g(\tau)(v_1 \wedge \cdots \wedge v_m) := \tau(v_1) \wedge \cdots \wedge \tau(v_m) $$

\begin{definition}
    Let $\pi : A \to B$. A \textbf{section} $g : B \to A$ is a map such that $\pi \circ g = id_B$
\end{definition}

$g : G \to \Gamma$ implicitly defined above is a section, and a homomorphism, such that $f_1 \circ g = id_{GL(A_1)}$. Indeed, check this composition for $\tau \in G$:
$$
f_1(g(\tau)) = (g(\tau))_1 = \tau
$$
Since $g(\tau)$ restricts $\tau$ on $A_1$

Let $N = ker(f_1)$, then we have short exact sequence $\mathcal{S}$:
$$
\begin{tikzcd}[row sep = small, column sep = small]
1 \arrow{r} & N \arrow{r}{\phi} & \Gamma \arrow[r, shift left=0.5ex, "f_1"] & \arrow[l, shift left=0.5ex, "g"] G \arrow{r}{t} & 1
\end{tikzcd}
$$

\begin{definition}
A sequence 
$$
\begin{tikzcd}[row sep = small, column sep = small]
\cdots \arrow{r} & X_1 \arrow{r}{f_1} & X_2 \arrow{r}{f_2} & X_3 \arrow{r} & \cdots
\end{tikzcd}
$$
is called \textbf{exact} at $X_2$ if $im(f_1) = ker(f_2)$.
\end{definition}
\begin{definition}
A sequence of groups is:
\begin{itemize}
    \item \textbf{exact} if it is exact at every position 
    \item \textbf{short} if it begins and ends with 1 (the trivial group), instead of continuing indefinitely like a long exact sequence
\end{itemize}
\end{definition}

The following proposition presents the following conditions for a sequence to be short and exact. In particular, we give the proposition on the sequence $\mathcal{S}$ to highlight ...

The proposition below provides the conditions for a sequence to be short exact. We state it here, in particular for the sequence $\mathcal{S}$ to serve our purpose.
\begin{proposition}
    A sequence $\mathcal{S}$ is short exact, if $\phi$ is injective, $f_1$ is surjective and $im(\phi) = ker(f_1)$
\end{proposition}
\begin{proof}
    In a typical short sequence like $\mathcal{S}$, the map $1 \to N$ has image ${id_N}$ (the identity element of $N$). 
    \begin{itemize}
        \item If $\phi$ is injective, then $ker(\phi) = {id_N} = im(1 \to N)$ so we have exactness at $N$.
        \item The trivial homomorphism $t$ takes $G$ to $1$, so all of $G$ is the kernel for $t$. If $f_1$ is surjective, then we also have $Im(f_1) = G = ker(t)$, thus exactness at $G$.
        \item Finally, if $\phi$ injective and $N = ker(f_1)$, so $im(\phi) = N = ker(f_1)$ so $\mathcal{S}$ is exact in the middle, at $\Gamma$.
    \end{itemize}
    We define $\phi$ an injection and $f_1$ is surjective, so $\mathcal{S}$ is indeed a short exact sequence.
\end{proof}

\begin{definition}
    Let $N, Q$ be groups, and let $\varphi : Q \to Aut(N)$ be a group homomorphism describing an action of $Q$ on $N$, namely conjugation. A \textbf{semidirect product} of $N$ and $Q$ with respect to $\varphi$ is the group $G = N \rtimes_\varphi Q$ which has the multiplying set $N \times Q$ (ordered pairs $(n, q)$) equipped with multiplication 
    $$
    (n_1, q_1) \cdot (n_2, q_2) = (n_1 \cdot \varphi(q_1)(n_2), q_1q_2)
    $$ 
    The identity is $(1_N, 1_Q)$ and the inverses are in form: $(n, q)^{-1} = (\varphi(q^{-1})(n^{-1}), q^{-1})$. Copy of $N$ is normal in the semidirect product.
\end{definition}

\begin{theorem}
    The short exact sequence $\mathcal{S}$ that splits (there exists a section $g$)yields the semidirect product decomposition $$\Gamma = N \rtimes G$$    
\end{theorem}
\begin{proof}
    There is a section $g : G \to \Gamma$ which is a group homomorphism so $\Gamma$ is split, meaning we can rebuild $\Gamma$ from $N$ and $G$. We demonstrate this as follows:

    \begin{lemma}
        With section $g$, every $\sigma \in \Gamma$ can be written uniquely as $$\sigma = n \cdot g(\tau)$$
        with $n \in N, \tau \in G$.
    \end{lemma}
    \begin{proof}
        Take $\sigma \in \Gamma$. Let $\tau := f_1(\sigma) \in G$
        Since $f_1(g(\tau)) = \tau$,\hspace{1ex} $\sigma$ and $g(\tau)$ have the same projection to $G$. Therefore 
        $$
        f_1(\sigma \cdot g(\tau)^{-1}) = f_1(\sigma) \cdot f_1(g(\tau))^{-1} = \tau \cdot \tau^{-1} = 1 
        $$ 
        So $\sigma \cdot g(\tau)^{-1} \in N = ker(f_1)$ Let $n := \sigma \cdot g(\tau)^{-1} $ be that element. We rewrite $\sigma$ as $$\sigma = n \cdot g(\tau), \quad n \in N, \tau \in G$$
        It remains to establish uniqueness of this decomposition. Suppose $$n_1 \cdot g(\tau_1) = n_2 \cdot g(\tau_2)$$
        Applying $f_1$ to both sides, we get $\tau_1 = \tau_2$.
        Cancelling $g(\tau_1)$ from the right, we get: $n_1 = n_2$
    \end{proof}
    
    Finally, the multiplication in semidirect product decomposition of $\Gamma$ is given as follows:
    Let $(n_1, \tau_1), (n_2, \tau_2) \in N \rtimes G$
    \begin{align*}
        (n_1, \tau_1), (n_2, \tau_2) 
        &= (n_1 \;g(\tau_1)) \;(n_2 \;g(\tau_2)) \\
        &= n_1 \;g(\tau_1) \;n_2 \;g(\tau_2) \\
        &= n_1 \;g(\tau_1) \;n_2 \;(g(\tau_1)^{-1} \;g(\tau_1)) \;g(\tau_2) \\
        &= n_1 \;(g(\tau_1) \;n_2 \;g(\tau_1)^{-1}) \;g(\tau_1) \;g(\tau_2)  \\
        &= n_1 \;n_2' \;g(\tau_1) \;g(\tau_2) \in N \rtimes GL(A_1) = \Gamma
    \end{align*}
    ($(g(\tau_1) \;n_2 \;g(\tau_2)^{-1}) \in N$ Since $N$ is normal)
\end{proof}

\begin{remark}
    We identify decomposition of a $\sigma = n \cdot g(\tau)$ with the ordered pair $(n, \tau)$    
\end{remark}

We shall give a short description of the contents of subgroups $G$ and $N$.
\paragraph{G.}
    Recall that $G = GL(A_1)$. According to the semidirect product decomposition $\Gamma$, there lies a copy of $GL(A_1)$ inside $\Gamma$. This is justified, as we showed how every linear automorphism of $A_1 = V$ extends uniquely to an automorphism of the whole exterior algebra. Concretely, these are \emph{grade preserving} automorphisms - they send generators in $A_1$ to linear combinations of generators, and extend multiplicatively to wedge products.

    \begin{example}
         Let $V$ be a vector space with a basis $\{x_1, \cdots, x_n\}$. Let $f : V \to V$
         $$
            f(x_i) = \begin{cases}
                x_2 & \text{if } i = 1\\
                x_3 & \text{if } i = 2\\
                x_1 & \text{if } i = 3\\
                x_i & \text{otherwise}
            \end{cases}
        $$
        This is a grade preserving automorphism in $G$. It is easy to check inverse is well-defined.
    \end{example}

    \paragraph{N.}
    $N$ is the kernel of surjective map $f_1: \Gamma \to G$. This means that $N$ consists of automorphisms of exterior algebra that act trivially on $A_1$, but can \emph{twist} higher degree components.

    The automorphisms of $N$ are of the form
    $$
    x \mapsto x + \text{(higher order terms)}
    $$
    In other words, elements of $N$ send each generator $x \in A_1$ to $x + u(x)$ where $u(x) \in \sum_{i=2}^n A_i$.

    We wish to exemplify a derivation in $N$, but to do that first we must first solidify the concept of inner derivations.
    \begin{definition}
    Let $A$ be a ring, $D_a : A \to A$ be a linear map. Suppose this map satisfies the Leibniz Rule, such that: $D(ab) = D(a)b + aD(b)$. We call $D_a$ a derivation of $A$.
    \end{definition} 
    
    Note that this is a general and purely algebraic definition. Although shares similar properties, it is not to be treated as the differentiation from calculus. It is easily seen that a commutator in $A_{comm}$ is a derivation as defined above. We define the derivative of $a$ on $A$ as follows:
    \begin{align*}
    D_a : A &\to A \\
    x &\mapsto ax - xa
    \end{align*}
    
    It is clear that a map $D_a$ defined this way is equal to the commutator $[a, -]$, such that $D_a(x) = [a, x]$. 
    
    \begin{lemma}\label{lem:dernil}
    If $a, b \in a_{odd}$, then $D_a D_b = 0$ 
    \end{lemma}
    \begin{proof}
        Let $x \in A_k$ be homogeneous. If $k$ is even, then by \ref{prop:parity}
        $$
        D_b(x) = bx - xb = 0
        $$
        If $k$ is odd, then
        $$
        D_b(x) = bx - xb = bx - (-1)^{|b||x|}bx = 2bx
        $$
        which has even degree. But $D_a$ vanishes on even degree inputs when $a$ is odd, hence $$
        D_a(D_b(x)) = 0
        $$
        Thus $D_aD_b = 0$ for all $x$.
    \end{proof}

    In particular, it follows from lemma $\ref{lem:dernil}$ that $D_a^2 = 0$ for $a \in A_{odd}$.

    \begin{definition}
        Exponential of a derivation $D$, $exp(D)$ is defined formally as:
        $$
        exp(D) = \sum_{k=0}^\infty \frac{D^k}{k!}
        $$
    \end{definition}
    In our case, due to nilpotency $D_a^2 = 0$, the exponential truncates to $exp(D_a) = 1 + D_a$.

    \begin{proposition}
        $exp(D_a)$ is an automorphism of the exterior algebra.
    \end{proposition}
    \begin{proof}
        Explicitly, for $x, y \in A$,

        \begin{align*} 
            (1+D_a)(xy)
            &= xy + D_a(xy)\\
            &= xy + D_a(x)y + xD_a(y)\\
            &= xy + D_a(x)y + xD_a(y) + D_a(x)D_a(y) & \text{by lemma \ref{lem:dernil}, $D_a(x)D_a(y) = 0$}\\
            &= (1+D_a)(x) (1+D_a)(y)
        \end{align*}
        Hence $exp(D_a) = 1 + D_a \in \Gamma$
    \end{proof}

    Ultimately, for $a \in A_{odd}$, 
    $$
    (1 + D_a)(x) = x + ax - xa \quad \forall x \in A
    $$
    This implies that $exp(D_a)$ fixes $A_1$ in $M_1/M_2$.
    Therefore it is clear that $exp(D_a)$ in $N$.\\

    Inside $N$, there is a nice abelian subgroup $N_1$ generated by exponentials of inner derivations $1 + D_a$ for $a \in A_{odd}$ These are called the \emph{inner automorphisms} of the exterior algebra. A reader interested in another semidirect product decomposion, which is $\Gamma = N_1 \rtimes A_0$ with much more explicitly defined subgroups $N_1$ and $A_0$ that isolates the inner automorphism structure may refer to full classification result in \cite{djokovic}.

\pagebreak

\subsection{Classification of Invariant Subalgebras}
In continuation, two questions naturally arise: 
\begin{enumerate}
    \item Are there any proper (non-trivial) invariant subspaces?
    \item If there are, -for classification- what are they?
\end{enumerate}
First question is answered positively with the next theorem.

\begin{remark}
    We recall the following identity of wedge product:
    $$A_i A_j \subseteq A_{i + j}$$
\end{remark}

It is proven in \cite{zahra}, that commutative polynomial rings do not have any nontrivial invariant subalgebras. However, we have:

\begin{lemma} \label{lem:comminv}
Let $A$ be {a} connected graded noncommutative algebra. Then the algebra generated by  commutators in $A$ is an invariant subalgebra, which is nontrivial. We {denote} this algebra by  $A_{comm}$.
\end{lemma}
\begin{proof}
  Let $G : = \{ [a, b]: = ab - ba \vert a, b \in A\}$, i.e. the set of commutators in $A$. Since $A$ is not commutative, there exist $a, b \in A$ such that $ab \neq ba$. Therefore, $ab - ba \neq 0 \in G$. Now let $A_{comm}$ be the algebra generated by $G$, that means that  intersection of all subalgebras of $A$ that contains $G$ or equivalently $A_{comm}$ is the set of all $\kk$-linear combinations of multiplication of elements of $G \cup \{ 1\}$.  Clearly $A_{comm}$ is invariant subalgebra, indeed for every $ f \in Aut_{\kk}(A)$, we have $ f([a, b]) = [f(a) , f(b)]$. 
It is enough to see that 
$A_{comm}$ is not trivial. Recall that $A = A_0 \oplus A_1 \oplus A_2 \oplus \dots$ and we suppose that $A_1$ is non zero. Then if  
$a, b \in A$, we write 
$a = a_0 + a_1 + \cdots a_n$, for some $n \geq 0$ and $a_i \in A_i$. Similarly $ b = b_0 + b_1 + \cdots b_m $, for some $m \geq 0$,  where $b_i \in A_i$. Then since $A_0 = \kk$, we have  
$ab - ba \in A_{2} + A_3 + A_4 + .... $.
That means $ G \subseteq A_{2} + A_3 + A_4 + \cdots  $.  Therefore $A_{comm}  \subseteq K+ 
A_2 + A_3 + \cdots < A$ 
\end{proof}

Consequently, due to existence of at least one non-trivial invariant subspace, exterior algebra is not a polynomial ring over $\kk$. This brings us to the second question: what are other invariant subalgebras of exterior algebras? In this section, we expibit this classification for the exterior algebra.

\begin{lemma}
    We charatcerize the center of $A$ as follows: 
    $ Z(A):= \left\{ a \in A \; \middle | \; ab = ba,\, \forall b \in A \right\} = \bigoplus _{i = 0} ^{\lfloor n/2 \rfloor} A_{2i}  +  A_n$
\end{lemma}
\begin{proof}
    Clearly every element in $ A_{2j}$ is in the center and if $T \in A_n$, then $T a = 0 = aT$, for every $a \in A$. This shows that  $\bigoplus _{i = 0} ^{\lfloor n/2 \rfloor} A_{2i}  +  A_n \subseteq Z(A)$. On the other hand, if 
    $x \in Z(A)$, then 
    we write $ x = x_1 \oplus x_2$, where $x_1 \in \bigoplus _{i = 0} ^{\lfloor n/2 \rfloor} A_{2i} \bigoplus A_n$ and $x_2 = 0$ or  $x_2 \in A \setminus \bigoplus _{i = 0} ^{\lfloor n/2 \rfloor} A_{2i} \bigoplus A_n $. Sicne $ xa = ax 
    $ and $ x_1 a = a x_1$, for every $a \in A$. We see that $ x_2 \in Z(A)$. We need to show that $x_2 = 0$. If not, then we write $ x_2 = \sum _{i = 0} ^{ \lfloor n/2 \rfloor - 1 } a_{2i+ 1} $.
    Let $j$ be the minimum $i$ such that $a_i \neq 0$. 
    Since $ j < n$ exists $e_k$ that $ a_j  e_k \neq 0 $. 
    Since for  every  $a_i$, we have $ a_i e_k = 0 = e_k a_i$ or $ a_i e_k = (-1) e_k a_i $, we conclude that 
    $ x_2 e_k = (-1) e_k x_2$ which is in contradiction with this fact that $ x_2 \in Z(A)$. So $x_2 = 0$. 
\end{proof}

\begin{remark}
    We underline that when $n$ is even, $Z(A) = A_{even}$
\end{remark}

\begin{proposition}
For any algebra $A$, the center of $A$ is an invariant subalgebra.
\end{proposition}
\begin{proof}
Indeed, if $f \in Aut_\kk (A)$ and $z \in Z(A)$, then, then for every $b \in A$:
$$ f(z) b = 
f(z) f (f^{- 1} (b) ) = f (z f^{- 1} (b) ) = f ( f^{- 1} (b) z  ) = b f(z) $$
concluding $f(z) \in Z(A)$. 
\end{proof}

\begin{lemma} \label{lem:commeven}
    $A_{comm},$ defined in Lemma \ref{lem:comminv}, is equal to
    $A_{even} = \bigoplus _{i = 0} ^{[n/2]} A_{2i}$. 
\end{lemma}
\begin{proof}   
    Let us show the inclusion $A_{even} \subseteq A_{comm}$. Let $m \geq 2 $ be even, then $ x_{i_1} \wedge \cdots \wedge x_{i_m} = 
    1/2 [ x_{i_1} \wedge \cdots \wedge x_{i_{m -1}}, x_{i_m} ]$. This implies that $A_m \subseteq A_{comm}$ and consequently $A_{even} \subseteq A_{comm}$. For the reverse inclusion, if $ a, b \in A$, then 
    by writing $a, b $ in form of $a = \sum _{i = 0} ^n a_i $ and  $b = \sum _{i = 0} ^n b_i $, where $a_i, b_i \in A_i$, we see that $a_i b_ j = b_j a_i$, whenever $ i $ or $j$ is 
    even. Therefore, $ [a_i, b_j] = \sum
    _{i + j}2 a_i a_j$, where $ i+ j$ is even.  
\end{proof}


\begin{corollary}\label{cor:eveninv}
    By \ref{lem:comminv} and equality of $A_{even}$ to $A_{comm}$ established in \ref{lem:commeven}, $A_{even}$ is a nontrivial invariant subalgebra.
\end{corollary}

Let us provide another type of invariant subalgebra.
\begin{proposition}
    Let $A$ be the Exterior algebra generated by $ \{x_1, \cdots, x_n\}$. The following direct sum is an invariant subalgebra:
        $B_{i} := A_0 \oplus (\bigoplus_{j=i}^n A_j),$ where $ i \geq 1.$
\end{proposition}
\begin{proof}
    Let $f \in Aut_{\kk}(A)$. Note that $ x_i \wedge x_i = 0$ and $f$ is ring homomorphism, thus, we have that $ 0 = f(0) = f(x_i) \wedge f(x_i)$. This implies $f(x_i) \in \bigoplus_{i = 1} ^{n}  A_i  $. Therefore $ f(A_i) \subseteq \bigoplus_{i = j} ^{n}  A_j $, for every $j \geq 1$. Indeed, from \ref{prop:basis} each $A_i$ has a finite basis in form $ \{ x_{j_1} \wedge \cdots \wedge x_{j_i} \mid \;k\text{-multi-indices } I\} $. Hence $B$  is invariant. 
\end{proof}

With the last proposition, we have now at hand, two types of direct sums of subspaces $A_i \in A$ that serve as invariant subalgebras. However, according to the classification of invariant subalgebras given by Demir and Nazemian~\cite{zahra-konur}, these are not the only possible forms an invariant subalgebra may possess. For completeness, we state the classification result here without proof.

\begin{notation}
    When  $1 \leq j \leq n$, define $P_j := \{\, S \subseteq \{ j, \dots, n \} \mid j \in S \,\}$. 
\end{notation}

\begin{theorem}
A nonzero subalgebra $B \subseteq {A}$ is invariant ($Aut$-stable) if and only if $B$ has one of the following forms:
\begin{itemize}
    \item[(a)] For some even $j > 0$,  
    \[
    B = \kk \oplus \bigoplus_{k \geq 0} A_{j+2k}.
    \]
\medskip
    \item[(b)]  There exist an odd $j$, a set $S \in P_j$, and an even $0 < i \leq j+1$ such that 
    \[
    \{\,s+i \mid s \in S, s+ i \leq n \,\}   \subseteq S
    \]
    and
    \[
    B = \bigoplus_{k \in S} A_{k} \;\oplus\; \bigoplus_{k \geq 0} A_{i+2k}.
    \]
\end{itemize}
\end{theorem}

\pagebreak

\printbibliography

@book{browne,
    author = {John Browne},
    title = {Grassmann Algebra Volume 1: Foundations: Exploring extended vector algebra with Mathematica},
    publisher = {Bernard Publishing},
    year = {2012}
}

@book{hungerford1980,
  author    = {Hungerford, Thomas W.},
  title     = {Algebra},
  publisher = {Springer-Verlag},
  year      = {2003},
  series    = {Graduate Texts in Mathematics},
  volume    = {73},
  address   = {New York},
  edition   = {12},
  isbn      = {9780387905181}
}

@online{nardozza2025,
  author       = {Vincenzo Nardozza},
  title        = {Talk 126, European Nonassociative Algebra Seminar},
  year         = {2025},
  month        = jun,
  day          = {30},
  organization = {European Nonassociative Algebra Seminar},
  url          = {https://youtu.be/5TpsZYvqoXo?si=jEi6B_Z-rakgd42h},
  note         = {University of Bari Aldo Moro, Italy},
}

@article{djokovic,
    author = {Dragomir Z. Djokovic},
    title = {Derivations And Automorphisms of Exterior Algebras},
    journal = {Canadian Journal of Math},
    volume = {30},
    number = {6},
    pages = {1336--1344},    
    year = {1978},
}

@article{desmond,
    author = {Desmond Fearnley-Sander},
    title = {Hermann Grassmann and the Creation of Linear Algebra},
    journal = {The American Mathematical Monthly},
    volume = {86},
    number = {10},
    pages = {809--817},
    year = {1979},
}

@unpublished{zahra,
  author = {Hongdi Huang and Zahra Nazemian and Yanhua Wang and James J. Zhang},
  title = {Relative Cancellation},
  note = {Preprint, \url{https://arxiv.org/abs/2503.10083}},
  year = {2025},
}

@unpublished{zahra-konur,
    author = {Mithat Konuralp Demir and Zahra Nazemian},
    title = {Aut-stable subspaces of Grassmann algebras},
    year = {2025},
    url = {https://arxiv.org/pdf/2508.16945}
}

\end{document}